\documentclass{amsart}

\usepackage{amssymb}
\usepackage{amsmath}
\usepackage{amsthm}

\newtheorem {theorem} {Theorem}[section]
\newtheorem {prop} [theorem]{Proposition}
\newtheorem {cory} [theorem]{Corollary}
\newtheorem {lemma}  [theorem]{Lemma}

\numberwithin{equation}{section}

\newcommand{\de}{\delta}
\newcommand{\la}{\lambda}
\newcommand{\si}{\sigma}
\newcommand{\sid}{\sigma_d}
\newcommand{\er}{\varepsilon}
\newcommand{\za}{\zeta}
\newcommand{\we}{\omega}
\newcommand{\ve}{v}

\newcommand{\ph}{\varphi}

\newcommand{\hol}{H}
\newcommand{\ii}{I}
\newcommand{\ff}{\Phi}
\newcommand{\Om}{\Omega}

\newcommand{\spd}{\partial B_d}

\newcommand{\cd}{{\mathbb{C}}^d}
\newcommand{\bd}{B_d}

\newcommand{\Nbb}{\mathbb N}
\newcommand{\Zbb}{\mathbb Z}
\newcommand{\rad}{\mathcal R}
\newcommand{\bloch}{\mathcal B}

\newcommand{\bw}{{\mathcal B}^{\omega}}
\newcommand{\bwp}{{\mathcal B}^{\omega}_p}
\newcommand{\sph}{\partial B}

\begin{document}

\title{Weighted Bloch spaces and quadratic integrals}

\author{Evgueni Doubtsov}

\address{St.~Petersburg Department
of V.A.~Steklov Mathematical Institute, Fontanka 27, St.~Petersburg
191023, Russia}

\email{dubtsov@pdmi.ras.ru}

\thanks{This research was supported by RFBR (grant no.~11-01-00526-a).}


\date{}


\begin{abstract}
Let $\mathcal{B}^{\omega}(B_d)$ denote the $\omega$-weighted Bloch space
in the unit ball $B_d$
of $\mathbb{C}^d$, $d\ge 1$.
We show that the quadratic integral
\[
\int_x^1 \frac{\omega^2(t)}{t}\, dt,\quad 0<x<1,
\]
governs the radial divergence
and integral reverse estimates in $\mathcal{B}^{\omega}(B_d)$.
\end{abstract}

\maketitle

\section{Introduction}\label{s_int}
Let $\hol(B_d)$ denote the space of holomorphic functions
on the unit ball $B_d$ of $\cd$, $d\ge 1$.

\subsection{Weighted Bloch spaces}\label{ss_growth_df}
Given a gauge function $\we: (0, 1]\to (0, +\infty)$,
the weighted Bloch space $\bw(B_d)$ consists of those $f\in\hol(B_d)$
for which
\begin{equation}\label{e_bw_def}
\|f\|_{\bw(\bd)} = |f(0)| + \sup_{z\in \bd}\frac{|\rad f(z)| (1-|z|)}{\we(1-|z|)}< \infty,
\end{equation}
where
\[
\rad f (z) = \sum_{j=1}^d z_j \frac{\partial f}{\partial z_j}(z), \quad z\in\bd,
\]
is the radial derivative of $f$.
$\bw(\bd)$ is a Banach space with respect to the norm defined by \eqref{e_bw_def}.
If $\we\equiv 1$, then $\bw(\bd)$ is the classical Bloch space $\bloch(\bd)$.
Usually we suppose that the gauge function $\we$ is increasing;
hence, we have $\bw(\bd) \subset \bloch(\bd)$.

The above notation is not completely standard: often the weight $t/ \we(t)$
is attributed to $\bw(\bd)$.

Assuming that $\we$ is sufficiently regular,
we show in the present paper that the quadratic integral
\[
\ii(x) = \ii_\we(x) =\int_x^1 \frac{\we^2(t)}{t}\, dt,\quad 0<x<1,
\]
governs the radial divergence and integral reverse estimates in $\bw(\bd)$.
In both cases, the solutions are based on the classical Hadamard gap series.

\subsection{Radial divergence}\label{ss_raddiv}
Given $f\in\hol(\bd)$ and $\za\in\spd$, we say
that $f$ has a radial limit at $\za$
if there exists a {\it finite}
limit $f^*(\za)= \lim_{r\to 1-} f(r\za)$.

Let $\sid$ denote the normalized Lebesgue measure on the unit sphere $\spd$.
The radial convergence or divergence in $\bw(\bd)$
is described in terms of $\ii(0+)$ by the following dichotomy:

\begin{prop}\label{p_raddiv}
Let $\we: (0, 1]\to (0, +\infty)$ be an increasing function.
\begin{itemize}
\item[(i)]
Let $\ii(0+)<\infty$.
If $f\in\bw(\bd)$, then $f$ has radial limits
$\sid$-almost everywhere.
\item[(ii)]
Let $\ii(0+)=\infty$ and let
$\we(t)/t^{1-\er}$ be decreasing for some $\er>0$.
Then the space $\bw(\bd)$ contains a function
with no radial limits $\sid$-almost everywhere.
\end{itemize}
\end{prop}

Remark that the condition $\ii(0+)=\infty$
was previously used by Dyakonov \cite{Dy02}
to construct a non-$\mathrm{BMO}$ function
lying in $\bw(B_1)$ and in all Hardy spaces $H^p(B_1)$,
$0<p<\infty$.

\subsection{Reverse estimates}\label{ss_reve}
Given an unbounded decreasing function $\ve: (0,1]\to (0,+\infty)$,
typical reverse estimates are obtained in the growth space
$\mathcal{A}^{\ve}(\bd)$, which consists
of $f\in\hol(\bd)$ such that $|f(z)|\le C\ve(1-|z|)$
for all $z\in\bd$.
Namely, under appropriate restrictions on $\ve$,
there exists a finite family $\{f_j\}_{j=1}^J \subset \mathcal{A}^{\ve}(\bd)$
such that
\[
|f_1(z)|+\dots +|f_J(z)| \ge C \ve(1-|z|)
\]
for all $z\in\bd$ (see, for example, \cite{AD12} and references therein).

For the weighted Bloch space $\bw(\bd)$,
the following result provides integral reverse estimates related
to the function $\ff^{\frac{1}{2}}(1-|z|)$, $z\in\bd$,
where
\[
\ff(x) = \ff_\we(x) = 1+ \int_x^1 \frac{\we^2(t)}{t}\, dt, \quad 0<x<1.
\]

\begin{theorem}\label{t_reve}
Let $d\in\Nbb$ and let $0<p<\infty$.
Assume that $\we: (0, 1]\to (0, +\infty)$ increases
and $\we(t)/t^{1-\er}$ decreases for some $\er>0$.
Then there exists
a constant $\tau_{d,p,\we}>0$ and
functions $F_{y}\in\bw(\bd)$, $0\le y\le 1$, such that
$\|F_{y}\|_{\bw(\bd)} \le 1$ and
\begin{equation}\label{e_reve}
\int_0^1 |F_{y}(z)|^{2p}\, dy \ge \tau_{d,p,\we}
\ff^{p}({1-|z|})
\end{equation}
for all $z\in\bd$.
\end{theorem}

For $\we\equiv 1$ and for logarithmic functions $\we$,
the above estimates were obtained in
\cite{DouPAMS} and \cite{Pe13}, respectively.

\subsection{Organization of the paper}
Section~\ref{s_raddiv} is devoted to the radial divergence problem.
In Section~\ref{s_reve}, we prove Theorem~\ref{t_reve}
and we show that estimate \eqref{e_reve} is sharp, up to a multiplicative constant.
Applications of Theorem~\ref{t_reve} are presented in Section~\ref{s_appl}.

\section{Radial divergence}\label{s_raddiv}
Proposition~\ref{p_raddiv}(i) is a known fact.
Indeed, if $\ii(0+)<\infty$ and $f\in\bw(\bd)$,
then $|\rad f(z)|^2 (1-|z|)$ is a Carleson measure,
hence, $f\in\textrm{BMOA}(\bd)$.
In particular, $f$ has radial limits $\sid$-a.e.

\subsection{Proof of Proposition~\ref{p_raddiv}(ii) for $d=1$}\label{ss_div1}
Put
\[
f(z) = \sum_{k=0}^{\infty} \we(2^{-k}) z^{2^k}, \quad z\in B_1.
\]
Standard arguments guarantee that $f\in\bw(B_1)$.
For example, let $t\in (0,1]$ and let $\tau=\frac{1}{t} \ge 1$.
Observe that
\[
\frac{\tau\we(\frac{1}{\tau})}{\tau}\quad\textrm{is a decreasing function of}\ \tau\ge 1,
\]
because $\we(t)$ is increasing. Also,
\[
\frac{\tau\we(\frac{1}{\tau})}{\tau^\er}\quad\textrm{is an increasing function of}\ \tau\ge 1,
\]
because $\we(t)/t^{1-\er}$ is decreasing. Therefore, $\tau\we(\frac{1}{\tau})$, $\tau\ge 1$,
is a normal weight in the sense of \cite{SW71}.
The derivative $f^\prime$ is represented by a Hadamard gap series, hence,
$f\in\bw(B_1)$ (see, e.g., \cite{Pa11}).

Since $\we$ is increasing, we have
\begin{equation}\label{e_we_div}
\sum_{k=0}^\infty \we^2(2^{-k}) \ge \ii(0+) = \infty.
\end{equation}
Thus, $f$ has no radial limits $\si_1$-a.e.
by \cite[Chapter~V, Theorem~6.4]{Zy59}.

\subsection{Proof of Proposition~\ref{p_raddiv}(ii) for $d\ge 2$}\label{ss_div2}

Fix a Ryll--Wojtaszczyk sequence $\{W[n]\}_{n=1}^\infty$ (see \cite{RW83}).
By definition, $W[n]$ is a holomorphic homogeneous polynomial of degree $n$,
$\|W[n]\|_{L^\infty(\spd)} =1$ and
$\|W[n]\|_{L^2(\spd)} \ge\de$ for a universal constant $\de>0$.
In particular, \eqref{e_we_div} guarantees that
\[
\sum_{k=0}^\infty \|\we(2^{-k}) W[2^k]\|^2_{L^2(\spd)} = \infty.
\]
Hence, by \cite[Lemma~7.2.7]{Ru80}, there exists a sequence $\{U_k\}_{k=1}^\infty$
of unitary operators on $\cd$ such that
\begin{equation}\label{e_za_div}
\sum_{k=0}^\infty \we^2(2^{-k}) |W[2^k]\circ U_k (\za)|^2 = \infty.
\end{equation}
for $\sid$-almost all $\za\in\spd$.
Put
\[
f(z) = \sum_{k=0}^{\infty} \we(2^{-k}) W[2^k]\circ U_k (z), \quad z\in \bd.
\]

First, fix a point $\za\in\spd$ with property \eqref{e_za_div}.
Consider the slice-function $f_\za(\la) =f(\la\za)$, $\la\in B_1$.
Note that
\[
f_\za(\la) = \sum_{k=0}^\infty a_k \la^{2^k},\quad \la\in B_1,
\]
where $a_k = \we(2^{-k}) W[2^k]\circ U_k(\za)$.
By \eqref{e_za_div}, we have $\{a_k\}_{k=1}^\infty \notin \ell^2$,
thus, $f_\za$ has no radial limits $\si_1$-a.e.
by \cite[Chapter~V, Theorem~6.4]{Zy59}.
Since the latter property holds for $\sid$-almost all $\za\in\spd$,
Fubini's theorem guarantees that $f$ has no radial limits $\sid$-a.e.

Second, recall that $\|W[2^k]\circ U_k\|_{L^\infty(\spd)} =1$.
So, we deduce that $f\in\bw(\bd)$,
applying the argument from Section~\ref{ss_div1}
to the slice-functions $f_\za$, $\za\in\spd$.
This ends the proof of Proposition~\ref{p_raddiv}.

\subsection{Comments}
\subsubsection{Radial divergence everywhere}
If $\we(0+)>0$, then $\bw(\bd)$
coincides with $\bloch(\bd)$,
hence, $\bw(\bd)$ contains
a function with no radial limits \textit{everywhere}
(see \cite{U88, U94}).
However, if $\we(0+)=0$, then Proposition~\ref{p_raddiv}(ii)
is not improvable in this direction.
Indeed, if $\we(0+)=0$ and $f\in\bw(B_1)$,
then $f$ has radial limits on a set of Hausdorff dimension one (see \cite{Mak89}).

\subsubsection{Hyperbolic setting}
To obtain the hyperbolic analog of $\bw(\bd)$,
replace $\rad f(z)$
by
\[
\frac{\rad \ph(z)}{1-|\ph(z)|^2},
\]
where $\ph: B_n\to B_m$, $m,n\in\Nbb$, is a holomorphic mapping.
The radial limit $\ph^*(\za)$ is defined at $\si_n$-almost every point of $\sph_n$,
hence, it is natural to replace the radial divergence condition by the following property:
$|\ph^*|=1$ $\sid$-a.e., that is, $\ph$ is inner.
While the problem in the hyperbolic setting is more sophisticated,
the following analog of Proposition~\ref{p_raddiv} is known,
at least for $n=m=1$.

\begin{theorem}[{\cite[Theorem~1.1]{Sm98}}, {\cite[Theorem~5.1]{AAN99}}]\label{t_aan_smith}
Let $\we: (0, 1]\to (0, +\infty)$ be an increasing function.
\begin{itemize}
\item[(i)]
Assume that $\ii(0+)<\infty$ and
$\ph: B_1\to B_1$ is a holomorphic function such that
\[
\frac{|\ph^\prime(z)| (1-|z|)}{1-|\ph(z)|} \le \we(1-|z|), \quad z\in B_1.
\]
Then $\ph$ is not inner.
\item[(ii)]
Assume that $\ii(0+)=\infty$ and
$\we(t)/t^{1-\er}$ decreases for some $\er>0$.
Then there exists an inner function $\ph: B_1\to B_1$ such that
\[
\frac{|\ph^\prime(z)| (1-|z|)}{1-|\ph(z)|} \le \we(1-|z|), \quad z\in B_1.
\]
\end{itemize}
\end{theorem}

In Section~\ref{ss_hyd},
we apply Theorem~\ref{t_reve} to obtain quantitative
versions of Theorem~\ref{t_aan_smith}(i).

\section{Reverse estimates}\label{s_reve}

\subsection{Auxiliary results}\label{ss_arw}
\begin{lemma}\label{l_reve}
Let $\we: (0, 1]\to (0, +\infty)$ be an increasing function. Put
\[
\Psi(r) = \sum_{k=0}^\infty \we^2(2^{-k}) r^{2^k -1}, \quad 0\le r <1.
\]
Then $\Psi(r) \ge C \ff(1-r)$ for a constant $C=C_\we >0$.
\end{lemma}
\begin{proof}
Let $2^{-n-1} \le 1-r < 2^{-n}$ for some $n\in\Zbb_+$. Then
\begin{align*}
2\Psi(r)
&\ge 2\we^2(1) + \sum_{k=1}^n \we^2(2^{-k}) \left(1- 2^{-n} \right)^{2^k -1}
\\
&\ge \we^2(1) + \frac{1}{e}\sum_{k=0}^n \we^2(2^{-k}) \ge C\ff(2^{-n-1})
\ge C\ff(1-r),
\end{align*}
since $\we$ is increasing and $\ff$ is decreasing.
\end{proof}

Also, we need the following improvement of the
Ryll--Wojtaszczyk theorem used in Section~\ref{ss_div2}.

\begin{theorem}[{\cite[Theorem~4]{Aab86}}]\label{t_ARW}
Let $d\in\Nbb$. Then there exist $\de=\de(d)\in (0,1)$ and
$J=J(d)\in\Nbb$ with the following property: For every $n\in\Nbb$,
there exist holomorphic homogeneous polynomials $W_{j}[n]$ of degree
$n$, $1\le j \le J$, such that
\begin{align}
 \|W_{j}[n]\|_{L^\infty(\sph_d)}
&\le 1\quad \textrm{and}\label{e_Linfty_ab} \\
 \max_{1\le j\le J} |W_{j}[n](\xi)|
&\ge \de\quad \textrm{for all\ } \xi\in\sph_d.\label{e_minmax_ab}
\end{align}
\end{theorem}

Probably, it is worth mentioning that $J(1)=1$.

\subsection{Proof of Theorem~\ref{t_reve}}\label{ss_bloch_rvs}
Let the constant $\de\in (0,1)$ and the
polynomials $W_{j}[n]$, $1\le j \le J$, $n\in\Nbb$, be those
provided by Theorem~\ref{t_ARW}.

For each non-dyadic $y\in [0,1]$, consider the following functions:
\[
F_{j, y} (z) = \sum_{k=0}^\infty R_k(y) \we(2^{-k}) W_j [2^k -1](z), \quad z\in
B_d,\ 1\le j \le J,
\]
where
\[
R_k (y) = \mathrm{sign}\sin (2^{k+1}\pi y), \quad y\in [0,1],
\]
is the Rademacher function.

First, arguing as in Section~\ref{s_raddiv} and using estimate~\eqref{e_Linfty_ab},
we deduce that
\linebreak
$\|F_{j,y}\|_{\bw(\bd)} \le C$.

Second, we obtain
\[
C_p \int_0^1
|F_{j,y}(z)|^{2p}\, dy
\ge
\left(\sum_{k=0}^\infty |\we(2^{-k}) W_j[2^k-1](z)|^2
\right)^{p}
\]
by \cite[Chapter~V, Theorem~8.4]{Zy59}.
Given positive numbers $a_j$, $1\le j \le J=J(d)$, we have
\[
 \left( \sum_{j=1}^{J} a_j \right)^{p}
\le C_{d, p} \sum_{j=1}^{J} a_j^{p}.
\]
Hence,
\[
C_{d, p} \sum_{j=1}^{J}
 \int_0^1 |F_{j,y}(z)|^{2p}\, dy
\ge
\left(\sum_{k=0}^\infty \sum_{j=1}^{J} \we^2(2^{-k})|W_j[2^k-1](z)|^2
 \right)^{p}.
\]
Since $W_j[2^k-1]$, $1\le j \le J$, are homogeneous polynomials of degree $2^k-1$, we obtain
\begin{align*}
 \sum_{k=0}^\infty \sum_{j=1}^{J} \we^2(2^{-k}) |W_j[2^k-1](z)|^2
&\ge \de^2 \sum_{k=0}^\infty \we^2(2^{-k}) |z|^{2^{k+2}-2} \\
&\ge \de^2 C_\we \ff(1-|z|^2), \quad z\in B_d,
\end{align*}
by~(\ref{e_minmax_ab}) and Lemma~\ref{l_reve} with $r=|z|^2$.
So,
\[
C_{d,p} \sum_{j=1}^J \int_0^1 |F_{j,y}(z)|^{2p}\, dy \ge
 \left( \de^2 C_\we \ff(1-|z|^2) \right)^{p},\quad z\in B_d.
\]
Changing the indices of the functions $F_{j,y}$
and using a new variable of integration,
we may reduce the above sum of integrals to one integral over $[0,1]$.
So, it remains to verify that $C \ff(1-r^2)\ge \ff(1-r)$, $0\le r <1$.

First, if $0\le r \le \frac{2}{3}$, then $\ff(1-r) \le C_\we \le C_\we\ff(1-r^2)$
for a constant $C_\we >0$. Second, if $0<\er < \frac{1}{3}$, then
$\ff(\er) - \ff(2\er)\le \we^2(2\er) \le 3\ff(2\er)$,
because $\we$ is increasing.
Thus $\ff(1-r) \le 4\ff(1-r^2)$ for $\frac{2}{3} < r<1$.

The proof of Theorem~\ref{t_reve} is finished.

\subsection{Integral means}
To show that inequality~\eqref{e_reve} is sharp,
we estimate the integral means
\[
M_p(f, r) = \left( \int_{\spd} |f(r\za)|^p \, d\sid(\za) \right)^{\frac{1}{p}}, \quad 0<r<1,
\]
for the functions $f\in\bw(\bd)$.

For $\we\equiv 1$, the following result was obtained in
\cite{CmG84} and \cite{Mak85}.

\begin{prop}\label{p_means}
Let $0<p<\infty$ and let $f\in\bw(\bd)$.
Then
\begin{equation}\label{e_dire}
M_p(f, r)
\le C\|f\|_{\bw(\bd)} \ff^{\frac{1}{2}}(1-r), \quad 0<r<1,
\end{equation}
for a constant $C>0$.
\end{prop}
\begin{proof}
For $f\in\hol(\bd)$ and $0<r<1$, we have
\[
M_p(f, r) \le C|f(0)|
+ C \left(\int_{\spd} \left( \int_0^1 r^2|\rad f(r t\za)|^2 (1-t)\, dt\right)^{\frac{p}{2}}
\,d\sid(\za) \right)^{\frac{1}{p}}
\]
for a constant $C>0$; see, for example, \cite[Theorem~3.1]{AB88}.
If $f\in\bw(\bd)$, then, using the defining property \eqref{e_bw_def},
we obtain
\[
\begin{split}
\int_0^1 r^2 |\rad f(r t\za)|^2 (1-t)\, dt
&= \int_0^r |\rad f(t\za)|^2 (r-t)\, dt
\\
&\le \|f\|^2_{\bw(\bd)} \int_0^r \frac{\we^2(1-t)}{1-t}\, dt
\le\|f\|^2_{\bw(\bd)}\ff(1-r).
\end{split}
\]
Since $|f(0)|\le \|f\|_{\bw(\bd)}$,
in sum we obtain the required estimate.
\end{proof}

Comparing Proposition~\ref{p_means} and Theorem~\ref{t_reve},
we conclude that the direct estimate \eqref{e_dire}
and the reverse estimate \eqref{e_reve} are not improvable,
up to multiplicative constants.

\subsection{Hardy--Bloch spaces}
Given a gauge function $\we$,
the weighted Hardy--Bloch space $\bwp(\bd)$, $0<p<\infty$,
consists of those $f\in\hol(\bd)$ for which
\begin{equation}\label{e_hb_def}
\|f\|_{\bwp(\bd)} = |f(0)| + \sup_{0<r<1} \frac{M_p(\rad f, r)(1-r)}{\we(1-r)}
<\infty.
\end{equation}

Clearly, we have $\bw(\bd) \subset \bwp(\bd)$, $0<p<\infty$.
So, it is interesting that estimate \eqref{e_dire} is sharp
for $f\in\bw(\bd)$ and holds for all $f\in\bwp(\bd)$
with $p\ge 2$.
Namely, we have the following proposition
that was proved in \cite{GPP06} for $\we\equiv 1$.

\begin{prop}\label{p_means_hb}
Let $2\le p<\infty$ and let $f\in\bwp(\bd)$.
Then
\begin{equation}\label{e_dire_hb}
M_p(f, r)
\le C\|f\|_{\bw(\bd)} \ff^{\frac{1}{2}}(1-r), \quad 0<r<1,
\end{equation}
for a constant $C>0$.
\end{prop}
\begin{proof}
For $f\in\hol(\bd)$ and $0<r<1$, we have
\begin{equation}\label{e_means_HL}
M_p(f, r) \le C|f(0)|
+ C \left(\int_0^1 \left( \int_{\spd} |\rad f(r t\za)|^p \,d\sid(\za)\right)^{\frac{2}{p}}
r^2(1-t)\, dt \right)^{\frac{1}{2}}
\end{equation}
for a constant $C>0$ (see \cite{HL37} for $d=1$; integration by slices gives the result for $d\ge 2$).
Now, we argue as in the proof of Proposition~\ref{p_means}. Namely,
for $f\in\bw(\bd)$, the defining property \eqref{e_hb_def} guarantees that
\[
\begin{split}
\int_0^1 \left( \int_{\spd} |\rad f(r t\za)|^p \,d\sid(\za)\right)^{\frac{2}{p}}
r^2(1-t)\, dt
&= \int_0^r M_p^2 (\rad f, t) (r-t)\, dt
\\
&\le \|f\|^2_{\bw(\bd)} \int_0^r \frac{\we^2(1-t)}{1-t}\, dt
\\
&\le\|f\|^2_{\bw(\bd)}\ff(1-r).
\end{split}
\]
Since $|f(0)|\le \|f\|_{\bw(\bd)}$,
the proof is finished.
\end{proof}

\section{Applications of Theorem~\ref{t_reve}}\label{s_appl}
In this section, we assume that
$\we: (0, 1]\to (0, +\infty)$ is an increasing function.

\subsection{Carleson measures}
Given a space $X\subset \hol(\bd)$ and $0<q<\infty$, recall that
a positive Borel measure $\mu$ on $\bd$ is called
$q$-Carleson for $X$ if
$X \subset L^q(\bd, \mu)$.

Suppose that $\we(t)/t^{1-\er}$ decreases for some $\er>0$.
A direct application of Theorem~\ref{t_reve} gives the following result:

\begin{cory}\label{c_carl}
Let $0<q<\infty$ and let $\mu$
be a $q$-Carleson measure for $\bw(\bd)$. Then
\[
\int_{\bd} \ff^{\frac{q}{2}} (1-|z|) \, d\mu(z) < \infty.
\]
\end{cory}

If $\mu$ is a radial measure, then the above corollary is reversible.
Moreover, the corresponding result holds for all spaces $\bwp(\bd)$, $p\ge 2$.

\begin{prop}\label{p_carl}
Let $0<q<\infty$ and let $\rho$ be a positive measure on $[0,1)$.
Then the following properties are equivalent:
\begin{align}
\int_0^1 \int_{\spd} |f(r\za)|^q\, d\sid(\za)\, d\rho(r)
&< \infty
\quad\textrm{for all\ } f\in\bwp(\bd),\ p\ge 2; \label{e_carl_rad1p}
\\
\int_0^1 \int_{\spd} |f(r\za)|^q\, d\sid(\za)\, d\rho(r)
&< \infty
\quad\textrm{for all\ } f\in\bw(\bd); \label{e_carl_rad1}
\\
\int_0^1 \ff^{\frac{q}{2}}(1-r) \, d\rho(r)
&< \infty. \label{e_carl_rad2}
\end{align}
\end{prop}
\begin{proof}
The implication \eqref{e_carl_rad1p}$\Rightarrow$\eqref{e_carl_rad1}
is trivial, because $\bw(\bd) \subset \bwp(\bd)$.
Next, \eqref{e_carl_rad1} implies \eqref{e_carl_rad2}
by Corollary~\ref{c_carl}.
Finally, Proposition~\ref{p_means_hb} guarantees that
\eqref{e_carl_rad2} implies \eqref{e_carl_rad1}.
\end{proof}

\subsection{Hyperbolic derivatives}\label{ss_hyd}
Let $\ii_\we(0+)<\infty$.
As observed in \cite{DouIEOT}, the conclusion of Theorem~\ref{t_aan_smith}(i)
remains true if the restriction
\[
\frac{|\ph^\prime(z)| (1-|z|)}{1-|\ph(z)|} \le \we(1-|z|), \quad z\in B_1,
\]
is replaced by the following weaker assumption:
\[
\frac{|\ph^\prime(z)| (1-|z|)}{1-|\ph(z)|}
\Om(1-|\ph(z)|) \le \we(1-|z|), \quad z\in B_1,
\]
where $\Om: (0, 1]\to (0, +\infty)$ is a bounded measurable function such that
\[
\ii_\Om = \int_0^1 \frac{\Om^2(t)}{t}\, dt =\infty.
\]
To obtain quantitative results of the above type, we apply Theorem~\ref{t_reve}.
Also, we make weaker assumptions about $\ph$.

So, suppose that $\Om$ is increasing and $\Om(t)/t^{1-\er}$ is decreasing for some $\er>0$. Put
\[
\ff_\Om(x) = 1 + \int_x^1 \frac{\Om^2(t)}{t}\, dt, \quad 0<x<1.
\]

\begin{cory}\label{c_ieot}
Let $\ph: B_1\to B_1$ be a holomorphic mapping and let $1\le p < \infty$.
Assume that $\ii_\we(0+)<\infty$, $\ii_\Om =\infty$ and
\begin{equation}\label{e_ieot}
(1-r) \left(\int_{\sph_1} \left( \frac{|\ph^\prime(r\za)| }{1-|\ph(r\za)|}\Om(1-|\ph(r\za)|) \right)^{2p}
\, d\si_1(\za) \right)^{\frac{1}{2p}}
 \le \we(1-r)
\end{equation}
for $0<r<1$. Then
\[
\sup_{0<r<1} \int_{\sph_1} \ff_\Om^{p}(1-|\ph(r\za)|)\, d\si_1(\za) <\infty.
\]
In particular, $|\ph^*|<1$ $\si_1$-a.e.
\end{cory}
\begin{proof}
Let the constant $\tau = \tau_{1, p, \Om}>0$
and the functions $F_y\in\bloch^\Om(B_1)$, $0\le y \le 1$,
be those provided by Theorem~\ref{t_reve}
for $d=1$ and for $\Om$ in place of $\we$.

Since $\|F_y\|_{\bloch^\Om(B_1)} \le 1$, we have
\[
|(F_y \circ\ph)^\prime (z)|\le |F_y^\prime(\ph(z))| |\ph^\prime(z)|
\le \frac{|\ph^\prime(z)|}{1-|\ph(z)|} \Om(1-|\ph(z)|), \quad z\in B_1.
\]
So, using \eqref{e_ieot} and the hypothesis $\ii_{\we}(0+)< \infty$,
we obtain
\[
\int_{0}^{1} M_{2p}^2 ((F_y\circ\ph)^\prime, t) (1-t)\, dt
\le \int_0^1 \frac{\we^2(1-t)}{1-t}\, dt < \infty.
\]
We further observe that $|F_y\circ \ph(0)| \le C_\ph \|F_y\|_{\bloch^\Om(B_1)} \le C$,
and so estimate \eqref{e_means_HL} guarantees that
\[
\int_{\sph_1} |F_y \circ \ph(r\za)|^{2p}\, d\si_1(\za)
\le C, \quad 0\le y\le 1,\ 0<r<1,
\]
for a universal constant $C>0$.
Hence, applying Fubini's theorem and Theorem~\ref{t_reve},
we obtain
\[
C \ge \int_{\sph_1}
\int_0^1 |F_y \circ \ph(r\za)|^{2p}\, dy \, d\si_1(\za)
\ge \int_{\sph_1} \ff_\Om^{p}(1-|\ph(r\za)|)\, d\si_1(\za),
\]
as required.
\end{proof}

\bibliographystyle{amsplain}

\end{document}